\documentclass[11pt]{amsart}
\usepackage{amsmath}
\usepackage{a4wide}
\usepackage[utf8]{inputenc}
\usepackage{amssymb}
\usepackage{amsopn}
\usepackage{epsfig}
\usepackage{amsfonts}
\usepackage{latexsym}
\usepackage{amsthm}
\usepackage{enumerate}
\usepackage[UKenglish]{babel}
\usepackage{verbatim}
\usepackage{color}
\usepackage{calrsfs}
\usepackage{pgf}
\usepackage{tikz}

\usepackage{mathrsfs}   % This provides an alternative way to typeset script capital letters, using \mathscr{}
 \usetikzlibrary{arrows,automata}

\DeclareMathAlphabet{\pazocal}{OMS}{zplm}{m}{n}
\newtheorem{theorem}{Theorem}[section]
\newtheorem{lemma}[theorem]{Lemma}
\newtheorem{proposition}[theorem]{Proposition}
\newtheorem{corollary}[theorem]{Corollary}
 \newtheorem{main}{Theorem}
  \newtheorem{cmain}[main]{Corollary}

\theoremstyle{definition}
\newtheorem{definition}[theorem]{Definition}

\newtheorem{conjecture}[theorem]{Conjecture}

\theoremstyle{remark}

\numberwithin{equation}{section}

\newcommand{\N}{\ensuremath{\mathbb{N}}}

\renewcommand{\c}{ {\mathbf{c}}}
\renewcommand{\d}{ {\mathbf{d}}}

\newcommand{\bb}{\mathcal{B}}

\newcommand{\K}{\widetilde{K}}

\newcommand{\set}[1]{\left\{#1\right\}}

\newcommand{\ga}{\gamma}

\newcommand{\ep}{\varepsilon}
\newcommand{\f}{\infty}
\newcommand{\de}{\delta}

\newcommand{\ee}{\mathcal{E}}
 
\newcommand{\lle}{\preccurlyeq}
\newcommand{\lge}{\succcurlyeq}

\newcommand{\si}{\sigma}

\newcommand{\ra}{\rightarrow}
\begin{document}

\title[Two bifurcation sets]{Two bifurcation sets arising from the beta transformation with a hole at $0$}
 \author{Simon Baker}
\address[S. Baker]{Mathematics institute, University of Warwick, Coventry, CV4 7AL, UK}
\email{simonbaker412@gmail.com}

\author{Derong Kong}
\address[D. Kong]{College of Mathematics and Statistics, Chongqing University,  401331, Chongqing, P.R.China}
\email{derongkong@126.com}

\date{\today}
\dedicatory{}

%\begin{frontmatter}

\subjclass[2010]{Primary: 37B10, Secondary: 28A78, 11A63}

\begin{abstract}
  Given $\beta\in(1,2],$ the $\beta$-transformation $T_\beta: x\mapsto \beta x\pmod 1$ on the circle $[0, 1)$ with a hole $[0, t)$ was investigated by Kalle et al.~(2019). They described the set-valued bifurcation set 
  \[
  \mathcal E_\beta:=\{t\in[0, 1): K_\beta(t')\ne K_\beta(t)~\forall t'>t\},
  \]
  where $K_\beta(t):=\{x\in[0, 1): T_\beta^n(x)\ge t~\forall n\ge 0\}$ is the survivor set. In this paper we investigate the dimension bifurcation set 
  \[
  \mathcal B_\beta:=\{t\in[0, 1): \dim_H K_\beta(t')\ne \dim_H K_\beta(t)~\forall t'>t\},
  \]
  where $\dim_H$ denotes the Hausdorff dimension.
  We show that if $\beta\in(1,2]$ is a multinacci number then the two bifurcation sets $\mathcal B_\beta$ and $\mathcal E_\beta$ coincide. Moreover we give a complete characterization of these two sets. As a corollary of our main result we prove that for $\beta$ a multinacci number we have $\dim_H(\mathcal E_\beta\cap[t, 1])=\dim_H K_\beta(t)$ for any $t\in[0, 1)$. This confirms a conjecture of Kalle et al.~for $\beta$ a multinacci number.
\end{abstract}

\keywords{Bifurcation sets; beta transformation; local dimension; survivor set.}
\maketitle

\section{Introduction}\label{s1}
Given $\beta\in(1,2]$, the $\beta$-transformation $T_\beta$ on the circle $\mathbb R/\mathbb Z\sim [0, 1)$ is defined by 
\[
T_\beta: [0, 1)\ra [0, 1);\quad x\mapsto \beta x\pmod 1.
\]
Following the pioneering work of R\'{e}nyi \cite{Renyi_1957} and Parry \cite{Parry_1960} there has been a great interest in the study of $T_\beta$. In general, the system $\Phi_\beta=([0, 1), T_\beta)$ does not admit a Markov partition, this makes describing the dynamics of $\Phi_\beta$ more challenging.  
%\textcolor{red}{I've removed some general discussion of the work of Schmeling and some others. It didn't seem to be relevant. Note it is still in the latex file and I have removed the related references.}

%Schmeling \cite{Schmeling-97} classified the set of $\beta\in(1,2]$ for which the underlying  $\beta$-shift satisfies certain properties. For example, he proved that the set of $\beta\in(1,2]$ for which the orbit $\{T_\beta^n(1)\}_{n=0}^\f$ is dense in $[0, 1)$ is residual and has full Lebesgue measure.  Li et al \cite{Li-Persson-Wang-Wu-14} investigated the Hausdorff dimension of the set of $\beta\in(1,2]$ for which the orbit $\{T_\beta^n(1)\}_{n=0}^\f$ approaches a given point $x_0\in[0, 1)$ at a given speed. Some more results on the Diophantine approximation properties of $T_\beta$ can be found in \cite{Bugeaud-Wang-2014, Lu-Wu-2016} and the references therein. 

When $\beta=2$,   Urba\'nski \cite{Urbanski_1986, Urbanski-87} considered the open dynamical system under the doubling map $T_2$ with a hole at zero. More precisely, for $t\in[0, 1)$ let 
\[
K_2(t):=\set{x\in[0, 1): T_2^n(x)\ge t~\forall \; n\ge 0}.
\]
He showed that the dimension function $t\mapsto \eta_2(t):=\dim_H K_2(t)$ is a Devil's staircase on $[0, 1)$, in particular $\eta_2$ satisfies the following properties: (i) $\eta_2$ is decreasing and continuous on $[0, 1)$; (ii) $\eta_2$ is locally constant   almost everywhere on $[0, 1)$; and (iii) $\eta_2$ is not constant on $[0, 1)$. Here and throughout the paper $\dim_H$  denotes the Hausdorff dimension. Moreover, he investigated the bifurcation sets
\begin{align*}
\ee_2&:=\set{t\in[0, 1): K_2(t')\ne K_2(t)~\forall\; t'>t},\quad
\bb_2:=\set{t\in[0, 1): \eta_2(t')\ne \eta_2(t)~\forall\; t'>t}.
\end{align*}
Clearly, $\bb_2\subseteq\ee_2$. In \cite{Urbanski_1986} Urba\'nski showed that $\bb_2=\ee_2$, and  its topological closure $\overline{\bb_2}$ is a Cantor set, i.e., a non-empty compact set that has neither isolated   nor interior points. Furthermore, the following local dimension property was shown to hold:
$
\lim_{r\to 0}\dim_H(\ee_2\cap(t-r, t+r))=\eta_2(t)$ for all  $t\in \ee_2. $
Recently, Carminati and Tiozzo in \cite{Carminati-Tiozzo-17} showed that the local H\"older exponent of the dimension function $\eta_2$ at any $t\in \ee_2$ equals $\eta_2(t)$. 

Inspired  by the work of Urba\'nski \cite{Urbanski_1986, Urbanski-87}, Kalle et al. in ~\cite{Kalle-Kong-Langeveld-Li-18} considered the analogous problem for the $\beta$-transformation with a hole $[0, t)$. More precisely, for $t\in[0, 1)$ they investigated the survivor  set
\[
K_\beta(t):=\set{x\in[0, 1): T_\beta^n(x)\ge t~\forall\; n\ge 0},
\]
and showed that the dimension function $t\mapsto  \dim_H K_\beta(t)$ is also a Devil's staircase on $[0, 1)$. Furthermore, they characterized the \emph{set-valued bifurcation set} 
\[
\ee_\beta:=\set{t\in [0, 1): K_\beta(t')\ne K_\beta(t)~\forall \; t'>t},
\]
and proved that $\ee_\beta$ is a Lebesgue null set of  full Hausdorff dimension for any $\beta\in(1,2)$.  Interestingly, they showed that  $\ee_\beta$ contains infinitely many isolated points for Lebesgue almost every $\beta\in(1, 2)$. This is in contrast to the case where $\beta=2$ and $\ee_2$ has no isolated points.

Since for each $\beta\in(1,2)$ the dimension function $\eta_\beta: t\mapsto \dim_H K_\beta(t)$ is a Devil's staircase,   it is natural to consider the \emph{dimension bifurcation set} 
\[
\bb_\beta:=\set{t\in[0, 1): \eta_\beta(t')\ne \eta_\beta(t)~\forall\; t'>t}.
\]
This set records those $t$ for which the dimension function $\eta_\beta$ has a `change' within any right neighborhood. 
Since $\eta_\beta$ is continuous,  $\bb_\beta$ cannot have isolated points. On the other hand,  the set-valued bifurcation set $\ee_\beta$ contains (infinitely many) isolated points for Lebesgue almost every $\beta\in(1,2)$. So in general we cannot expect the coincidence of the two bifurcation sets 
$\bb_\beta$ and $\ee_\beta$. That being said, in this paper we show that if $\beta$ is a multinacci number, i.e.,   the unique root in $(1,2)$ of the equation
\[
x^{m+1}=x^m+x^{m-1}+\cdots+x+1
\] 
for some $m\in\N$, 
then the two bifurcation sets indeed coincide.

When $\beta\in(1,2)$ is a multinacci number, the following result for the set-valued bifurcation set $\ee_\beta$  was established in \cite[Theorems C and D]{Kalle-Kong-Langeveld-Li-18}. We record it here for later use.
\begin{theorem}[\cite{Kalle-Kong-Langeveld-Li-18}]\label{th:kkll}
Let $\beta\in(1,2]$ be  a multinacci number. Then the topological closure $\overline{\ee_\beta}$ is a Cantor set. Furthermore, $\max\overline{\ee_\beta}=1- 1/\beta$. 
\end{theorem}

In order to give a complete description of the dimension bifurcation set $\bb_\beta$  we introduce a class of basic intervals.
\begin{definition}\label{def:lyndon word-interval}
Let $\beta\in(1,2]$. A word 
  $s_1\ldots s_m$ is called \emph{$\beta$-Lyndon} if 
  \begin{align*}
  &s_{i+1}\ldots s_{m}\succ s_1\ldots s_{m-i}\quad \forall~ 1\le i<m,\quad\textrm{and}\quad
   \si^n((s_1\ldots s_m)^\f)\prec \de(\beta)\quad\forall ~n\ge 0.
  \end{align*} 
Accordingly, an interval $[t_L, t_R)\subset[0, 1)$ is called a \emph{$\beta$-Lyndon interval}  if there exists a $\beta$-Lyndon word $s_1\ldots s_m$  such that 
\[
t_L=\sum_{i=1}^m\frac{s_i}{\beta^i}\quad \textrm{and} \quad t_R=\frac{\beta^m }{\beta^m-1} \cdot t_L.
\]
\end{definition} In the above definition $\succ$ corresponds to the usual lexicographic ordering and $\delta(\beta)$ is the quasi-greedy $\beta$-expansion of $1.$ These are both defined formally in the next section.

We will show that the $\beta$-Lyndon intervals are pairwise disjoint for all $\beta\in(1,2]$, and when $\beta$ is multinacci they cover the interval $[0, 1-1/\beta)$ up to a Lebesgue null set.
The latter statement can be seen as a consequence of our main result for the coincidence of the two bifurcation sets, which we state below. 
\begin{main}\label{main:1}
Let $\beta\in(1,2]$ be a multinacci number. Then

\begin{align*}
\bb_\beta=\ee_\beta&=\left[0,1-\frac{1}{\beta}\right)\setminus\bigcup[t_L, t_R)\\
&=\set{t\in[0, 1): \lim_{r\to 0}\dim_H(\bb_\beta\cap(t, t+r))=\dim_H K_\beta(t)>0},
\end{align*}
where the union is taken over all  pairwise disjoint $\beta$-Lyndon intervals. 
\end{main}

By Theorem \ref{main:1}  it follows that the topological closure $[t_L, t_R]$ of each $\beta$-Lyndon interval is indeed a  maximal interval where the dimension function $\eta_\beta$  is constant. As a corollary of Theorem \ref{main:1} we confirm a conjecture of \cite{Kalle-Kong-Langeveld-Li-18} for $\beta$ a multinacci number. 
\begin{cmain}
\label{cor:1}
If $\beta\in(1,2]$ is a multinacci number, then 
\[
\dim_H(\ee_\beta\cap[t, 1])=\dim_H K_\beta(t)\quad\forall ~t\in[0, 1). 
\]
\end{cmain}

The rest of the paper is organized as follows.  In the next section we recall some properties from symbolic dynamics and the dimension formula for the survivor set $K_\beta(t)$. The proof of Theorem \ref{main:1} and Corollary \ref{cor:1} will be given in Section \ref{sec:proof of th-1}. In the final section we make some remarks and point out that the method of proof for Theorem \ref{main:1} can be applied to some other special values of $\beta$.

 \section{Preliminaries and $\beta$-Lyndon intervals}\label{sec:symbolic-dimension}

 Given $\beta\in(1,2]$, for each $x\in I_\beta:=[0, 1/(\beta-1)]$ there exists a sequence $(d_i)=d_1d_2\ldots\in\set{0, 1}^\N$ such that 
 \[
 x=\sum_{i=1}^\f\frac{d_i}{\beta^i}=:((d_i))_\beta.
 \] 
The sequence $(d_i)$ is called a \emph{$\beta$-expansion} of $x$. Sidorov \cite{Sidorov_2003} showed that for $\beta\in(1,2)$ Lebesgue almost every $x\in I_\beta$ has a continuum of $\beta$-expansions. This is rather different from the case when $\beta=2$ where every number in $I_2=[0, 1]$ has a unique dyadic expansion except for countably many points that have precisely two expansions. Given $x\in I_\beta$, among all of its $\beta$-expansions  let $b(x, \beta)=(b_i(x, \beta))$ be the \emph{greedy} $\beta$-expansion of $x$, i.e., the lexicographically largest $\beta$-expansion of $x$. Such a sequence always exists and is generated by the orbit of $x$ under the map $T_{\beta}.$ Similarly, for $x\in(0, 1/(\beta-1)]$ let $a(x, \beta)=(a_i(x, \beta))$ be the \emph{quasi-greedy} $\beta$-expansion of $x$ (cf.~\cite{Daroczy_Katai_1993}), which is the lexicographically largest $\beta$-expansion of $x$ not ending with $0^\f$. Here for a word $\c$ we denote by $\c^\f:=\c\c\cdots$ the periodic sequence with periodic block $\c$. 
 Throughout the paper we will use the lexicographic order between sequences and words in the usual way. For example, for two sequences $(c_i), (d_i)\in\set{0, 1}^\N$ we write $(c_i)\prec (d_i)$ if $c_1<d_1$, or there exists $n>1$ such that $c_1\ldots c_{n-1}=d_1\ldots d_{n-1}$ and $c_n<d_n$. Furthermore, for two words $\c, \d$ we say $\c\prec \d$ if $\c 0^\f\prec \d 0^\f$. 
 
 For $\beta\in(1,2]$ let
\[
\de(\beta)=\de_1(\beta)\de_2(\beta)\ldots
\]
be the quasi-greedy $\beta$-expansion of $1$, i.e., $\de(\beta)=a(1, \beta)$. Let $\si$ be the left-shift on $\set{0, 1}^\N$ defined by $\si((c_i))=(c_{i+1})$. Then $b(T_\beta(x), \beta)=\si(b(x, \beta))$ for any $x\in[0, 1)$.  The following lexicographic characterizations of $\de(\beta)$ and the greedy expansion $b(x, \beta)$ are essentially due to Parry \cite{Parry_1960} (see also \cite{DeVries_Komornik_2008}). 

  \begin{lemma}\label{lem:greedy and quasi-greedy}
  \begin{enumerate}[{\rm(i)}]
 
  \item The map $\beta\mapsto \de(\beta)$ is a strictly increasing bijection from $(1, 2]$ onto the set of sequences $(\de_i)\in\set{0, 1}^\N$ not ending with $0^\f$ and  satisfying
  $
  \si^n((\de_i)) \lle (\de_i) ~\forall ~n\ge 0.
 $
  
  \item Let $\beta\in(1,2]$. Then the map $x\mapsto b(x, \beta)$ is a strictly increasing bijection from $[0, 1)$ onto the set of all sequences $(b_i)\in\set{0, 1}^\N$ satisfying
  $
  \si^n((b_i)) \prec \de(\beta)~\forall ~n\ge 0.
  $ 
  
  \item For any $\beta\in(1,2)$ the sequence $b(1,\beta)=(b_i)$ satisfies $\si^n((b_i))\prec \de(\beta)~\forall~ n\ge 1$.  
   \end{enumerate}
  \end{lemma}
  For $\beta\in(1,2]$ let $[t_L, t_R)$ be a $\beta$-Lyndon interval generated by a $\beta$-Lyndon word $s_1\ldots s_m$. Then by Definition \ref{def:lyndon word-interval} and Lemma \ref{lem:greedy and quasi-greedy} (ii) it follows that 
 \[
 b(t_L, \beta)=s_1\ldots s_m 0^\f\quad \textrm{and}\quad  b(t_R, \beta)=(s_1\ldots s_m)^\f. 
 \]
   \begin{lemma}\label{lem:Lyndon-interval}
   
 For any $\beta\in(1,2]$ the $\beta$-Lyndon intervals are pairwise disjoint.
 \end{lemma}
 \begin{proof}
 Let $[t_L, t_R)$ and $[t_L', t_R')$ be two $\beta$-Lyndon intervals generated by the $\beta$-Lyndon words $s_1\ldots s_p$ and $s_1'\ldots s'_q$, respectively. Suppose on the contrary that $[t_L, t_R)\cap [t_L', t_R')\ne \emptyset$. Without loss of generality we assume $t_L<t_L'< t_R$. Then by Definition \ref{def:lyndon word-interval} and Lemma \ref{lem:greedy and quasi-greedy}(ii) it follows that 
 \begin{equation*} 
 s_1\ldots s_p0^\f\prec s'_1\ldots s'_q0^\f\prec (s_1\ldots s_p)^\f.
 \end{equation*}
 This implies $q>p$,  $s_1'\ldots   s_p'=s_1\ldots s_p$ and  
 $s'_{p+1}\ldots s_q'0^\f\prec (s_1\ldots s_p)^\f.$ Write $q=N p+r$ with $N\ge 1$ and $0<r\le p$.
 So, either there exists $1\le k<N$ such that 
 \[s_{p+1}'\ldots s_{kp}'=(s_1\ldots s_p)^{k-1}\quad\textrm{and}\quad s_{kp+1}'\ldots s_{(k+1)p}'\prec s_1\ldots s_p,\]
  or 
  \[s_{p+1}'\ldots s_{Np}'=(s_1\ldots s_p)^{N-1}\quad\textrm{and}\quad s_{Np+1}'\ldots s_q'\lle s_1\ldots s_{q-Np}.\]
   Using $s_1'\ldots s_p'=s_1\ldots s_p$ we conclude  in both cases that 
 \[
 s_{j+1}'\ldots s_q'\lle s_1'\ldots s_{q-j}'\quad\textrm{for some }j\in\set{p, p+1,\ldots, q-1}.
 \]
 This is not possible by the definition of a $\beta$-Lyndon word. 
 \end{proof}
  
To describe the Hausdorff dimension of the survivor set 
\[K_\beta(t)=\set{x\in[0, 1): T_\beta^n(x)\ge t~\forall n\ge 0},\]
we recall from \cite[Chapter 4]{Lind_Marcus_1995} the definition of topological entropy for a symbolic set. 
For a set $X\subset\set{0, 1}^\N,$ its \emph{topological entropy} is defined to be 
\[
h(X)=\liminf_{n\to\f}\frac{\log\#B_n(X)}{n},
\]
where $B_n(X)$ is the set of all length $n$ prefixes of sequences from $X$. 

The following characterization of the set-valued bifurcation set $\ee_\beta$ was implicitly given in \cite{Urbanski_1986} (see also \cite[Proposition 2.3]{Kalle-Kong-Langeveld-Li-18}). Furthermore,  the  Hausdorff dimension of $K_\beta(t)$ was implicitly given by Raith in \cite{Raith-89}, and was recently explicitly presented  in  \cite[Equation (2.6)]{Kalle-Kong-Langeveld-Li-18}.

\begin{proposition}\label{prop:property of dim-K(t)}
\begin{enumerate}[{\rm(i)}]
\item  Let $\beta\in(1,2]$. Then 
 \[
 \ee_\beta=\set{t\in[0, 1): T_\beta^n(t)\ge t~\forall n\ge 0}.
 \]

\item Let $\beta\in(1,2]$ and   $t\in[0, 1)$. Then the Hausdorff dimension of $K_\beta(t)$ is given by 
 \begin{equation*}\label{eq:dimension-K}
 \dim_H K_\beta(t)=\frac{h(\K_\beta(t))}{\log\beta},
 \end{equation*}
  where 
 $
 \K_\beta(t):=\set{(x_i)\in\set{0,1}^\N: b(t, \beta)\lle \si^{n}((x_i))\lle \de(\beta)\; \forall n\ge 0}.
 $
 Furthermore, the dimension function $\eta_\beta: t\mapsto \dim_H K_\beta(t)$ is a Devil's staircase, i.e., $\eta_\beta$ is a non-constant, decreasing and  continuous function which is  locally constant almost everywhere in $[0, 1)$.  
 \end{enumerate}
\end{proposition}

\section{Proof of Theorem \ref{main:1}}\label{sec:proof of th-1}

 In this section we will prove Theorem \ref{main:1}. First we show that the dimension bifurcation set $\bb_\beta$ coincides with the set-valued bifurcation set $\ee_\beta$, we then derive a complete characterization of these sets via the $\beta$-Lyndon intervals. The proof heavily relies upon the transitivity of the symbolic survivor set $\K_\beta(t)$ (see Lemma \ref{lem:transitivity} below).

\begin{proposition}\label{prop:coincidence}
Let $\beta\in(1,2)$ be  a multinacci number. Then 
\[
\bb_\beta=\ee_\beta=\left[0,1-\frac{1}{\beta}\right)\setminus\bigcup[t_L, t_R),
\]
where the union is taken over all $\beta$-Lyndon intervals. 
\end{proposition}
Observe by Lemma \ref{lem:Lyndon-interval} that the $\beta$-Lyndon intervals are pairwise disjoint. In fact the closed $\beta$-Lyndon intervals $\set{[t_L, t_R]}$ are also pairwise disjoint. So by Proposition \ref{prop:coincidence} it follows that each closed $\beta$-Lyndon interval is a maximal interval where the dimension function $\eta_\beta$ is constant. 

The proof of Proposition \ref{prop:coincidence} will be split into several lemmas.
We fix a multinacci number   $\beta\in(1,2)$ with  $\de(\beta)=(1^m0)^\f$ for some $m\ge 1$. In view of Proposition \ref{prop:property of dim-K(t)}  it is necessary to investigate the symbolic survivor set 
\[\K_\beta(t)=\set{(x_i)\in\set{0, 1}^\N: b(t, \beta)\lle \si^n((x_i))\lle \de(\beta)~\forall n\ge 0}.\]

 \begin{lemma}\label{lem:transitivity}
 Let $\beta\in(1,2)$  with $\de(\beta)=(1^m0)^\f$, and let $[t_L, t_R)\subset[0, 1-1/\beta)$ be a $\beta$-Lyndon interval. Then the set-valued map $t\mapsto \K_\beta(t)$ is constant on   $[t_L, t_R]$, and the set   $\K_\beta(t_R)$ is a transitive subshift of finite type.  
 \end{lemma}
 \begin{proof}
 Suppose    $[t_L, t_R)$ is a $\beta$-Lyndon interval generated by $s_1\ldots s_p$. Take $t\in[t_L, t_R]$. Then by Lemma \ref{lem:greedy and quasi-greedy}(ii) it follows that  
 \[\K_\beta(t_R)\subseteq \K_\beta(t)\subseteq \K_\beta(t_L).\]
  Observe that $\de(\beta)=(1^m0)^\f$ for some $m\in\mathbb N$. Then 
 \begin{equation}\label{eq:K-constant}
 \begin{split}
  \K_\beta(t_L)&=\set{(x_i): s_1\ldots s_p0^\f\lle \si^n((x_i))\lle (1^m 0)^\f\; \forall n\ge 0}\\
  &=\set{(x_i): (s_1\ldots s_p)^\f\lle \si^n((x_i))\lle (1^m 0)^\f\; \forall n\ge 0}=\K_\beta(t_R).
  \end{split}
  \end{equation}
  Here we use the simple argument that 
  \[\si^n((x_i))\lge s_1\ldots s_p 0^\f ~\forall n\ge 0\quad\Longleftrightarrow\quad \si^n((x_i))\lge (s_1\ldots s_p)^\f ~\forall n\ge 0.\] 
  So, the set-valued map $t\mapsto \K_\beta(t)$ is constant on $[t_L, t_R]$. Furthermore, $\K_\beta(t_R)$ is
 a subshift of finite type with forbidden blocks 
$c_1\ldots c_k$ satisfying  $c_1\ldots c_k0^\f\prec s_1\ldots s_p0^\f$ or $c_1\ldots c_k0^\f\succ (1^m0)^\f,$
 where $k= \max\set{p, m+1}$. It remains to prove the transitivity of $\K_\beta(t_R)$.

Since $[t_L, t_R)\subset[0, 1-\frac{1}{\beta})$, by Lemma \ref{lem:greedy and quasi-greedy} (ii) it follows that $b(t_R, \beta)\prec b(1-\frac{1}{\beta}, \beta)$, which gives
 \begin{equation}\label{eq:kong-0}
(s_1\ldots s_p)^\f\prec 01^m 0^\f.
 \end{equation}
Arbitrarily fix an admissible word $\ep=\ep_1\ldots \ep_k$ and an admissible sequence $\ga=\ga_1\ga_2\ldots$ in $\K_\beta(t_R)$.  We will construct   a word $\nu$ such that $\ep\nu\ga\in\K_\beta(t_R)$.   Observe that $\si^n((s_1\ldots s_p)^\f)\prec (1^m 0)^\f$ for all $n\ge 0$. Thus, there exists  a large integer $N$   such that 
 \begin{equation}\label{eq:kong-1}
 \si^n((s_1\ldots s_p)^\f)\prec (1^m 0)^N0^\f\quad\textrm{for all}\quad n\ge 0.
 \end{equation}
 
 Denote by  $(\de_i):=\de(\beta)=(1^m0)^\f$. Note that $\ep_{i+1}\ldots \ep_{k}\lle \de_1\ldots \de_{k-i}$ for all $0\le i<k$. Let $i_0\in\set{0,1,\ldots, k-1}$ be the smallest index such that  $\ep_{i_0+1}\ldots \ep_{k}=\de_1\ldots \de_{k-i_0}$. If such an index $i_0$ does not exist, then we put $i_0=k$. In either case there exists a word $\mu$ such that $\ep\mu=\ep_1\ldots\ep_{i_0}(1^m0)^N$.  Since $\ga\lle(1^m0)^\f$, there exists $q\in\set{0,1,\ldots, m}$ such that  $\ga$ begins with $\ga_1\ldots\ga_{q+1}=1^q0$.  We claim that 
$
\ep\mu1^{m-q}\ga=\ep_1\ldots\ep_{i_0}(1^m0)^{N+1}\ga_{q+2}\ga_{q+3}\ldots\;\in\K_\beta(t_R),
$
 or equivalently, 
 \begin{equation}\label{eq:kong-2}
(s_1\ldots s_p)^\f\lle \si^n(\ep\mu 1^{m-q}\ga)\lle (1^m0)^\f\quad\textrm{for all }n\ge 0.
 \end{equation}
 
 First we prove the second inequality in (\ref{eq:kong-2}).
 By the definition of $i_0$ it follows that $\si^n(\ep\mu 1^{m-q}\ga)\prec \de(\beta)= (1^m0)^\f$ holds for all $0\le n<i_0$. Furthermore, since $\ga\in\K_\beta(t_R)$, the second inequality in (\ref{eq:kong-2}) also holds for $n\ge |\ep|+|\mu|+m-q$.   For the remaining $n$ we observe  that 
$\si^{i_0}(\ep\mu 1^{m-q}\ga)=(1^{m}0)^{N+1}\ga_{q+2}\ga_{q+3}\ldots$ and $\ga_{q+2}\ga_{q+3}\ldots\in\K_\beta(t_R)$. So    $\si^n(\ep\mu 1^{m-q}\ga)\lle (1^m0)^\f$ for all $i_0\le n<|\ep|+|\mu|+m-q$. This proves the second inequality in (\ref{eq:kong-2}).

For the first inequality in  (\ref{eq:kong-2}) 
we observe that
 $\ep\mu 1^{m-q}\ga=\ep_1\ldots \ep_{i_0}(1^m 0)^N 1^m\ga_{q+1}\ga_{q+2}\ldots$ {and} $ \ga_{q+1}\ga_{q+2}\ldots\in\K_\beta(t_R).$ Then by  (\ref{eq:kong-0}) it follows that $\si^n(\ep\mu 1^{m-q}\ga)\lge(s_1\ldots s_p)^\f$ for all $n\ge i_0$. If $i_0=0$, then we are done. Otherwise,  we take $0\le n<i_0$.  Since $\ep_1\ldots \ep_{i_0}$ is an admissible word in $\K_\beta(t_R)$, we have 
 \[
 \ep_{n+1}\ldots \ep_{i_0}\lge t_1\ldots t_{i_0-n},
 \]
 where $(t_i):=(s_1\ldots s_p)^\f$. The first inequality in (\ref{eq:kong-2}) now holds by (\ref{eq:kong-1}), which tells us that 
   \[(1^m0)^{N}1^m \ga_{q+1}\ga_{q+2}\ldots\succ t_{i_0-n+1}t_{i_0-n+2}\ldots.\] 
     This completes the proof of our claim.
  
Since $\ep$ and $\ga$ are chosen arbitrarily, it follows that $\K_\beta(t_R)$ is transitive.  
  \end{proof}
  
 To prove the coincidence of $\bb_\beta$ and $\ee_\beta$ we still need the following inequalities. 
 \begin{lemma}\label{lem:basic inequality}
 Let $(t_1\ldots t_N)^\f\in\set{0,1}^\N$ be a periodic sequence with   period $N\ge 2$. If 
 \[\si^n((t_1\ldots t_N)^\f)\lge (t_1\ldots t_N)^\f\quad \forall ~n\ge 0,\] then 
 \[
 t_{j+1}\ldots t_N\succ t_1\ldots t_{N-j}\quad \forall ~1\le j<N.
 \]
 \end{lemma}
 \begin{proof}
 Note that $N\ge 2$ is the  period of $(t_1\ldots t_N)^\f$, and
  \begin{equation}\label{eq:nov6-1}
  \si^n((t_1\ldots t_N)^\f)\lge (t_1\ldots t_N)^\f\quad \forall ~n\ge 0.
  \end{equation}
  Then  $t_1=0$ and $t_N=1$. 
   Taking the reflection on both sides of (\ref{eq:nov6-1}) it follows that 
  \[
  \si^n((\overline{t_1\ldots t_N})^\f)\lle (\overline{t_1\ldots t_N})^\f\quad\textrm{for all }n\ge 0.
  \]
  Here for a word $c_1\ldots c_k\in\set{0, 1}^k$ its reflection is defined by $\overline{c_1\ldots c_k}:=(1-c_1)(1-c_2)\ldots (1-c_k)$. 
 By Lemma \ref{lem:greedy and quasi-greedy}(i) it follows that $(\overline{t_1\ldots t_N})^\f$ is the quasi-greedy expansion of $1$ for some $\beta'\in(1,2]$, i.e., $\de(\beta')=(\overline{t_1\ldots t_N})^\f$. Since $N$ is the  period of the sequence $\de(\beta')$, we have $b(1, \beta')=\overline{t_1\ldots t_{N-1}}\,1 0^\f$.  So, by Lemma \ref{lem:greedy and quasi-greedy} (iii) it follows that  
  \[
  \overline{t_{j+1}\ldots t_N}\prec\overline{t_{j+1}\ldots t_{N-1}}\,1\lle \overline{t_1\ldots t_{N-j}}
  \] {for all }$1\le j<N.$
 Then the lemma follows by taking the reflection in the above equation.
 \end{proof}

Now we prove the coincidence of the two bifurcation sets.
\begin{lemma}\label{lem:inclusion-C}
Let $\beta\in(1,2)$ with $\de(\beta)=(1^m0)^\f$. Then $\ee_\beta=\bb_\beta$. 
\end{lemma}
\begin{proof}
By the definition of the two bifurcation sets it is easy to see that $\bb_\beta\subset\ee_\beta$. So in the following we prove $\ee_\beta\subset\bb_\beta$. 

Let $t\in \ee_\beta$ with $b(t, \beta)=(t_i)$. Then  by Theorem \ref{th:kkll} we have $t\le 1-1/\beta<1/\beta$. This gives $t_1=0$. By Lemmas \ref{lem:greedy and quasi-greedy} (ii) and Proposition \ref{prop:property of dim-K(t)} (i) it follows that  $\si^n((t_i))\lge (t_i)$ for all $n\ge 0$.  Let $N\ge 1$ be the smallest index such that $\si^N((t_i))=(t_i)$. If such an integer $N$ does not exist, then we set $N=\f$. In the following we will prove $t\in\bb_{\beta}$ by considering the following two cases: (I) $N<\f$; and (II) $N=\f$.
  
  Case (I). $N<\f$. We claim that 
  $t_1\ldots t_N$  is a $\beta$-Lyndon word. If $N=1$, then $(t_i)=t_1^\f=0^\f$. It is easy to check that $t_1=0$ is a $\beta$-Lyndon word. In the following we assume $N\ge 2$. 
Since $\si^N((t_i))=(t_i)$, we have   $(t_i)=(t_1\ldots t_N)^\f$. Note that $(t_i)$ is the greedy $\beta$-expansion of $t$. Then by Lemma \ref{lem:greedy and quasi-greedy} (ii) it follows that $\si^n((t_1\ldots t_N)^\f)\prec \de(\beta)$ for all $n\ge 0$. Note that $\si^n((t_1\ldots t_N)^\f)\lge(t_1\ldots t_N)^\f$. Then by Lemma \ref{lem:basic inequality} and the definition of $N,$ it follows that 
  \begin{equation*}
  t_{j+1}\ldots t_N\succ t_1\ldots t_{N-j}\quad\textrm{for all }1\le j<N.
  \end{equation*} 
  So  by Definition \ref{def:lyndon word-interval} we establish the claim.  
  
Hence, $t=((t_1\ldots t_N)^\f)_\beta=t_R$ is the right endpoint of a $\beta$-Lyndon interval generated by $t_1\ldots t_N$. By Lemma \ref{lem:transitivity} it follows that $\K_\beta(t)$ is a transitive subshift of finite type. Observe that for any $t'>t$ we have $\K_\beta(t')\subset\K_\beta(t)$ and     $(t_1\ldots t_N)^\f\in\K_\beta(t)\setminus\K_\beta(t')$.  
Recall by \cite[Corollary 4.4.9]{Lind_Marcus_1995}  that for any transitive subshift of finite type, any proper subshift has strictly smaller topological entropy. Therefore, 
\[h(\K_\beta(t'))<h(\K_\beta(t))\quad \textrm{for any}\quad t'>t.\]
By Proposition \ref{prop:property of dim-K(t)} (ii) this yields $\eta_\beta(t')<\eta_\beta(t)$ for any $t'>t$.  So $t\in\bb_\beta$.   

Case (II). $N=\f$. Then $\si^n((t_i))\succ(t_i)$ for all $n\ge 1$. So $(t_i)$ is not periodic. Observe that  $(t_i)$ begins with digit $0$, and $\si^n((t_i))\prec (1^m0)^\f$ for all $n\ge 0$.   So there exists a subsequence $(m_k)$ of positive integers such that for any $k\ge 1$ we have $t_{m_k}=0$,   and the word $t_1\ldots t_{m_k}^+:=t_1\ldots t_{m_k-1}1$ does not contain $m+1$ consecutive ones. 
Then by noting $t_1=0$ it follows that 
\[\si^n((t_1\ldots t_{m_k}^+)^\f)\prec (1^m 0)^\f\quad \forall ~n\ge 0.
\]  
Since $\si^n((t_i))\lge (t_i)$ for all $n\ge 0$, by Definition \ref{def:lyndon word-interval} it follows that $t_1\ldots t_{m_k}^+$ is a $\beta$-Lyndon word for any $k\ge 1$. Let $s_k:=((t_1\ldots t_{m_k}^+)^\f)_\beta$. Then $s_k$ is the right endpoint of a $\beta$-Lyndon interval generated by $t_1\ldots t_{m_k}^+$. Furthermore, $s_k$ strictly decreases to $t=((t_i))_\beta$ as $k\ra\f$.

So, for any  $t'>t$  we can find $k$ such that $s_k\in(t, t')$. By  the same arguments as in the proof of Case (I) for $s_k$ we conclude that 
\[
\eta_\beta(t')<\eta_\beta(s_k)\le \eta_\beta(t). 
\]
So $t\in\bb_\beta$. This completes the proof. 
\end{proof}

Finally, we describe the bifurcation sets via the $\beta$-Lyndon intervals.
 \begin{lemma}\label{lem:inclusion-B}
 Let $\beta\in(1,2]$ with $\de(\beta)=(1^m0)^\f$. Then 
 \[
 \left[0, 1-\frac{1}{\beta}\right)\setminus\bigcup[t_L, t_R)\subset\ee_\beta.
 \]
 \end{lemma}
 \begin{proof}
 Take  $t\in[0, 1-1/\beta)\setminus\ee_\beta$ with greedy $\beta$-expansion $(t_i)$. Then $t_1=0$. Since $t\notin \ee_\beta$,   by Proposition \ref{prop:property of dim-K(t)} (i) there exists a smallest  positive integer $N$ such that $T_\beta^N(t)<t$, which implies
  \begin{equation}\label{eq:dr-1}
  t_{N+1}t_{N+2}\ldots  \prec (t_i).
  \end{equation}
  We claim that $t_1\ldots t_N$ is a $\beta$-Lyndon word. Clearly, if $N=1$ then $t_1=0$ is a $\beta$-Lyndon word. In the following we assume $N\ge 2$. In view of Definition \ref{def:lyndon word-interval} it suffices to prove  
  \begin{equation}\label{eq:lyndon-ti}
  t_{j+1}\ldots t_N\succ t_1\ldots t_{N-j}\quad\textrm{for all }1\le j<N,
  \end{equation}
  and
  \begin{equation}\label{eq:admissible-ti}
  \si^n((t_1\ldots t_N)^\f)\prec (1^m 0)^\f\quad\textrm{for all }n\ge 0.
  \end{equation}

First we prove (\ref{eq:lyndon-ti}). By the definition of $N$ in (\ref{eq:dr-1}) it follows that 
  \begin{equation}\label{eq:dr-2}
  t_{j+1}t_{j+2}\ldots\lge (t_i)\quad\textrm{for all }1\le j<N,
  \end{equation}
  which implies $t_{j+1}\ldots t_{N}\lge t_1\ldots t_{N-j}$ for all $1\le j<N$. Suppose $t_{j+1}\ldots t_N=t_1\ldots t_{N-j}$ for some $j\in\set{1,2,\ldots, N-1}$. Then by (\ref{eq:dr-1}) and (\ref{eq:dr-2}) it follows that 
  \begin{align*}
  t_{j+1}t_{j+2}\ldots =t_1\ldots t_{N-j}t_{N+1}t_{N+2}\ldots\prec t_1\ldots t_{N-j}t_1t_2\ldots\lle (t_i),
  \end{align*}
  leading to a contradiction with the minimality of $N$. This proves (\ref{eq:lyndon-ti}). 
    
To prove (\ref{eq:admissible-ti}) we observe that $\de(\beta)=(1^m0)^\f$ and $(t_i)$ is the greedy $\beta$-expansion of $t$. Then    $t_1\ldots t_N$ cannot contain $m+1$ consecutive ones. Since $t_1=0$, we  have   $\si^n((t_1\ldots t_N)^\f)\lle (1^m 0)^\f$ for all $n\ge 0$. So to prove  (\ref{eq:admissible-ti}) it remains to prove that $\si^n((t_1\ldots t_N)^\f)\ne(1^m 0)^\f$ for any $n\ge 0$.   
  Suppose  the equality $\si^n((t_1\ldots t_N)^\f)=(1^m 0)^\f$ holds for some $n\ge 0$.   Then using $t_1=0$ it follows    $t_1\ldots t_{m+1}=01^m$.
  % \textcolor{red}{Do we need \eqref{eq:lyndon-ti}? Doesn't this follow simply from the equality and $t_1=0$? I know this doesn't change anything but it might simplify things.}. 
  This implies $(t_i)\lge 01^m0^\f=b(1-1/\beta, \beta)$. By Lemma \ref{lem:greedy and quasi-greedy} (ii) we have  $t\ge 1-1/\beta$, leading to a contradiction. So the claim follows. 
  
 By the claim there exists a $\beta$-Lyndon interval $[t_L, t_R)$ generated by $t_1\ldots t_N$.   Furthermore,  by (\ref{eq:dr-1}) it follows that 
  \begin{align*}
  (t_i)=t_1\ldots t_Nt_{N+1}t_{N+2}\ldots &\prec t_1\ldots t_N t_1t_2\ldots=(t_1\ldots t_N)^2t_{N+1}t_{N+2}\ldots\\
  &\prec (t_1\ldots t_N)^2t_1t_2\ldots=(t_1\ldots t_N)^3t_{N+1}t_{N+2}\ldots\\
  &\cdots\\
  &\lle (t_1\ldots t_N)^\f.
  \end{align*}
  Therefore, 
  $
  t_1\ldots t_N 0^\f\lle (t_i)\prec (t_1\ldots t_N)^\f,
  $
which gives $t\in[t_L, t_R)$ by Lemma \ref{lem:greedy and quasi-greedy} (ii).  This completes the proof.
 \end{proof}
 
  \begin{proof}[Proof of Proposition \ref{prop:coincidence}]
By Lemmas \ref{lem:inclusion-C} and \ref{lem:inclusion-B} it suffices to prove 
\[
\bb_\beta\subset\left[0, 1-\frac{1}{\beta}\right)\setminus\bigcup[t_L, t_R). 
\]
Note by Lemma \ref{lem:inclusion-C} and Theorem \ref{th:kkll} that $\bb_\beta=\ee_\beta\subset[0, 1-1/\beta]$. In fact we have $\ee_\beta\subset[0, 1-1/\beta)$.   Note that $b(1-1/\beta, \beta)=01^m0^\f$. So $T^{m+1}_\beta(1-1/\beta)<1-1/\beta$. By Proposition \ref{prop:property of dim-K(t)} (i) this implies $1-1/\beta\notin\ee_\beta$. Hence, $\ee_\beta\subset[0,1-1/\beta)$.

In the following it remains to prove $\bb_\beta\cap\bigcup[t_L, t_R)=\emptyset$.  
  If $t\in [t_L, t_R)$, then   by (\ref{eq:K-constant}) it follows that  $\K_\beta(t)= \K_\beta(t_L)= \K_\beta(t_R)$, which gives  $\eta_\beta(t')=\eta_\beta(t)=\eta_{\beta}(t_L)$ for all  $t'\in(t, t_R)$. So, $t\notin\bb_\beta$. This completes the proof.
    \end{proof}
  
As a consequence of Proposition \ref{prop:coincidence} and Theorem \ref{th:kkll} it follows that for $\beta\in(1,2]$  a multinacci number the $\beta$-Lyndon intervals cover $[0, 1-1/\beta)$ up to a Lebesgue null set. 
\begin{corollary}\label{cor:covers}
 Let $\beta\in(1,2]$ be a multinacci number.  
\begin{enumerate}[{\rm(i)}]
\item   The  union of all $\beta$-Lyndon intervals covers $[0, 1-1/\beta)$ up to a Lebesgue null set. Furthermore, for any $t\in\bb_\beta$ and any $r>0$ the interval  $(t, t+r)$ contains infinitely many $\beta$-Lyndon intervals. 

\item $\eta_\beta(t)>0$ if and only if $t<1-1/\beta$. 
\end{enumerate}
\end{corollary}  
\begin{proof}
Note by Theorem \ref{th:kkll} that $\ee_\beta$ is a Lebesgue null set with no isolated points. Then (i) follows from Proposition \ref{prop:coincidence} which tells us that $\bigcup[t_L, t_R)=[0, 1-1/\beta)\setminus\ee_\beta$. For (ii) it can be deduced from Proposition \ref{prop:coincidence} and Theorem \ref{th:kkll} that $\sup \bb_\beta=1-1/\beta$ and $1-1/\beta\notin\bb_\beta$. 
\end{proof}

 Now we turn to investigate the local dimension of the bifurcation set $\bb_\beta$. 
 \begin{lemma}
 \label{lem:local dimension}
 Let $\beta\in(1,2]$ with $\de(\beta)=(1^m0)^\f$. Then  
 \[
\lim_{r\to 0}\dim_H(\bb_\beta\cap(t, t+r))=\dim_H K_\beta(t)>0\quad\forall ~t\in\bb_\beta.
 \]
 \end{lemma}
 \begin{proof}
Take $t\in \bb_\beta$. By Proposition \ref{prop:coincidence} we have $t<1-1/\beta$, and then by Corollary \ref{cor:covers} (ii) it gives  $\eta_\beta(t)=\dim_H K_\beta(t)>0$.  
 Note by Proposition \ref{prop:coincidence} and Proposition \ref{prop:property of dim-K(t)} (i) that $
\bb_\beta\cap (t, t+r)=\ee_\beta\cap(t, t+r)\subseteq K_\beta(t)
$ for any $r>0$. Then 
$
\lim_{r\ra 0}\dim_H(\bb_\beta\cap(t, t+r))\le \eta_\beta(t).
$
 So it remains to prove 
 \begin{equation}\label{eq:kong-26}
\lim_{r\ra 0}\dim_H(\bb_\beta\cap(t, t+r))\ge\eta_\beta(t).
 \end{equation}
We prove this now by considering  the  following two cases: (I) $t=t_R$ is the right endpoint of a $\beta$-Lyndon interval; (II) $t\in[0, 1-1/\beta)\setminus\bigcup[t_L, t_R]$. 
 
Case (I). Suppose $t=t_R$ is the right endpoint of a $\beta$-Lyndon interval. 
Let $(t_i)=(t_1\ldots t_p)^\f$ be the greedy $\beta$-expansion of $t_R$. Note that $t_R\in \bb_\beta$. Then by Corollary \ref{cor:covers} (i) there exists a sequence $(t_R^{(n)})\subset \bb_\beta$ such that each $t_R^{(n)}$ is a right endpoint of a $\beta$-Lyndon interval and $t_R^{(n)}\searrow t_R$ as $n\ra\f$. Fix $r>0$. There exists a large integer $N$ satisfying $t_R^{(n)}\in(t_R, t_R+r)$ for all $n\ge N$. Furthermore, by Lemma \ref{lem:greedy and quasi-greedy} (ii) it follows that for each $n\ge N$ there exists an integer  $k_n$ such that the greedy $\beta$-expansion $b(t_R^{(n)}, \beta)$ of $t_R^{(n)}$ satisfies
  \begin{equation}\label{eq:kdr-sss}
  b(t_R^{(n)},\beta)\succ  (t_1\ldots t_p)^{k_n}1^\f. 
  \end{equation}
 Observe by Proposition \ref{prop:coincidence} and Proposition \ref{prop:property of dim-K(t)} (i) that 
  \[\bb_\beta=\ee_\beta=\set{((s_i))_\beta: (s_i)\lle \si^n((s_i))\prec (1^m0)^\f\; \forall n\ge 0}.\]
 So by using  $t_R\in \bb_\beta$,  (\ref{eq:kdr-sss}) and Lemma \ref{lem:greedy and quasi-greedy} (ii) it follows that for any $n\ge N$,
  \begin{equation}\label{eq:hx-1}
 \begin{split}
 \set{\big((t_1\ldots t_p)^{k_n}x_1x_2\ldots\big)_\beta:  x_1\ldots x_{p}=t_1\ldots t_p,~(x_i)\in \K_\beta(t_R^{(n)})}&\subseteq\bb_\beta\cap[t_R, t_R^{(n)})\\
 &\subseteq \bb_\beta\cap [t_R, t_R+r).
 \end{split}
  \end{equation}
  Note by Lemma \ref{lem:transitivity} that $\K_\beta(t_R^{(n)})$ is a transitive subshift of finite type. Then by (\ref{eq:hx-1}) it follows that 
    \[
  \dim_H(\bb_\beta\cap (t_R, t_R+r))\ge\dim_H K_\beta(t_R^{(n)})=\eta_\beta(t_R^{(n)})\quad\textrm{for all }n\ge N.
  \]
  Letting $n\ra\f$ and by the continuity of $\eta_\beta$ we obtain that
  \[
  \dim_H(\bb_\beta\cap(t_R, t_R+r))\ge \eta_\beta(t_R).
  \]
  Since $r>0$ was given arbitrary, letting $r\ra 0$ we conclude that 
  \begin{equation}\label{eq:hx-2}
  \lim_{r\ra 0}\dim_H(\bb_\beta\cap(t_R, t_R+r))\ge \eta_\beta(t_R).
  \end{equation}
  
Case (II).   $t\in [0, 1-\frac{1}{\beta})\setminus\bigcup[t_L, t_R]$. Then  by Corollary \ref{cor:covers} (i)  there exists a sequence $(t_R^{(k)})$ such that each $t_R^{(k)}$ is the right endpoint of a $\beta$-Lyndon interval, and $t_R^{(k)}\searrow t$ as $k\ra\f$. So, for any $r>0$ there exists a sufficiently large integer $k$ such that $t_R^{(k)}\in(t, t+r)$. By (\ref{eq:hx-2}) with $t_R$ replaced by $t_R^{(k)}$ it follows that for any $\ep>0$ there exists $r_k>0$ such that $(t_R^{(k)}, t_R^{(k)}+r_k)\subset(t, t+r)$ and then 
\[
\dim_H(\bb_\beta\cap(t, t+r))\ge\dim_H(\bb_\beta\cap(t_R^{(k)}, t_R^{(k)}+r_k))\ge  \eta_\beta(t_R^{(k)})-\ep.
\]
Letting $r\ra 0$, and then $t_R^{(k)}\ra t$, we conclude by the continuity of $\eta_\beta$ that 
\begin{equation*}
\lim_{r\ra 0}\dim_H(\bb_\beta\cap(t, t+r))\ge \eta_\beta(t)-\ep.
\end{equation*}
Since $\ep>0$ was arbitrary, we obtain $\lim_{r\ra 0}\dim_H(\bb_\beta\cap(t, t+r))\ge \eta_\beta(t)$.
This, together with (\ref{eq:hx-2}), proves   (\ref{eq:kong-26}).
 \end{proof}
 
  \begin{proof}[Proof of Theorem \ref{main:1}]
Let $\beta\in(1,2)$ with $\de(\beta)=(1^m0)^\f$. By Lemma \ref{lem:Lyndon-interval}, Proposition \ref{prop:coincidence} and Lemma \ref{lem:local dimension} it suffices to prove 
\begin{equation}\label{eq:bk-11}
 \set{t\in[0, 1): \lim_{r\to 0}\dim_H(\bb_\beta\cap(t, t+r))=\eta_\beta(t)>0}\subset\bb_\beta.
\end{equation}

Take $t\in[0, 1)\setminus\bb_\beta$. Then by Proposition \ref{prop:coincidence} we have $t\in[1-1/\beta, 1)$ or $t\in[t_L, t_R)$ for some $\beta$-Lyndon interval.  If $t\ge 1-1/\beta$, then $\eta_\beta(t)=0$ by Corollary \ref{cor:covers} (ii).  If $t\in[t_L, t_R)$, then there exists $r>0$ such that  $\bb_\beta\cap(t, t+r)=\emptyset$. This  completes the proof. 
\end{proof}

\begin{proof}[Proof of Corollary \ref{cor:1}]
Note by Proposition \ref{prop:coincidence} that $\ee_\beta\subset[0,1-1/\beta)$. So if $t\ge 1-1/\beta$, then clearly the result  holds by Corollary \ref{cor:covers} (ii). Now let $t\in[0, 1-1/\beta)$.  
Observe by Proposition \ref{prop:property of dim-K(t)} (i) that $\ee_\beta\cap[t,1]\subset K_\beta(t)$. So it suffices to prove 
\begin{equation}\label{eq:cor-1}
\dim_H(\ee_\beta\cap[t, 1])\ge \dim_H K_\beta(t).
\end{equation}
If $t\in[0, 1-1/\beta)\setminus[t_L, t_R)$, then (\ref{eq:cor-1}) follows by   Lemma \ref{lem:local dimension}. If $t\in[t_L, t_R)$, then we still have (\ref{eq:cor-1}) by using Lemma \ref{lem:local dimension} that 
\begin{align*}
\dim_H(\ee_\beta\cap[t, 1])\ge\dim_H(\ee_\beta\cap[t_R, 1])\ge\dim_H K_\beta(t_R)=\dim_H K_\beta(t),
\end{align*} 
where the last equality holds by (\ref{eq:K-constant}).
\end{proof}

\section{Final remarks}
The main results obtained in this paper can be easily modified to study the following analogous bifurcation sets:
\begin{align*}
&\ee_\beta':=\set{t\in[0, 1): K_\beta(t')\ne K_\beta(t)~\forall t'\ne t},\\
&\bb_\beta':=\set{t\in[0,1): \dim_H K_\beta(t')\ne\dim_H K_\beta(t)~\forall t'\ne t}.
\end{align*}
If $\beta\in(1,2]$ is a multinacci number, one can show that 
\begin{align*}
\bb_\beta'&=\ee_\beta'=\left[0, 1-\frac{1}{\beta}\right)\setminus\bigcup[t_L, t_R]\\
&=\set{t\in[0, 1): \lim_{r\to 0}\dim_H(\ee_\beta\cap(t-r, t))=\lim_{r\to 0}\dim_H(\ee_\beta\cap(t, t+r))=\dim_H K_\beta(t)>0},
\end{align*}
where the union is taken over all pairwise disjoint closed $\beta$-Lyndon intervals. 

Observe that the main result Theorem \ref{main:1}  holds under the assumption that $\beta\in(1,2]$ is a multinacci number, i.e., $\de(\beta)=(1^m0)^\f$ for some $m\in\N$. The method used in this paper can be adapted to show that Theorem \ref{main:1} still holds for $\beta\in(1,2]$ with $\de(\beta)=(10^m)^\f$. It is worth  mentioning that in \cite{Kalle-Kong-Langeveld-Li-18} Kalle et al.~considered a general Farey word base $\beta$, i.e., $\de(\beta)=(s_1\ldots s_p)^\f$ with $s_ms_{m-1}\ldots s_2s_1$ a non-degenerate Farey word. They showed that for a general Farey word base $\beta\in(1,2),$ the set-valued bifurcation set $\ee_\beta$ has no isolated points and Theorem \ref{th:kkll} holds. We finish by posing the following conjecture.
\begin{conjecture}
Let $\beta\in(1,2]$. Then $\bb_\beta=\ee_\beta$ if and only if $\ee_\beta$ has no isolated points. 
\end{conjecture}

\section*{Acknowledgements}
The authors were supported by an LMS Scheme 4 grant. The first author was supported by EPSRC grant EP/M001903/1. The second author was supported by NSFC No.~11401516, and by the Fundamental Research Funds for the Central Universities No.~2019CDXYST0015. He wishes to thank the Mathematical Institute of Leiden University.

% \bibliographystyle{abbrv}
%\bibliography{Fractal-Expansions}

\begin{thebibliography}{10}

%\bibitem{Bugeaud-Wang-2014}
%Y.~Bugeaud and B.-W. Wang.
%\newblock Distribution of full cylinders and the {D}iophantine properties of
%  the orbits in {$\beta$}-expansions.
%\newblock {\em J. Fractal Geom.}, 1(2):221--241, 2014.

\bibitem{Carminati-Tiozzo-17}
C.~Carminati and G.~Tiozzo.
\newblock The local {H}\"{o}lder exponent for the dimension of invariant
  subsets of the circle.
\newblock {\em Ergodic Theory Dynam. Systems}, 37(6):1825--1840, 2017.

\bibitem{Daroczy_Katai_1993}
Z.~Dar{\'o}czy and I.~K{\'a}tai.
\newblock Univoque sequences.
\newblock {\em Publ. Math. Debrecen}, 42(3-4):397--407, 1993.

\bibitem{DeVries_Komornik_2008}
M.~de~Vries and V.~Komornik.
\newblock Unique expansions of real numbers.
\newblock {\em Adv. Math.}, 221(2):390--427, 2009.

\bibitem{Kalle-Kong-Langeveld-Li-18}
C.~Kalle, D.~Kong, N.~Langeveld, and W.~Li.
\newblock The $\beta$-transformation with a hole at 0.
\newblock {\em arXiv:1803.07338. To appear in Ergodic Theory Dynam. Systems}.

%\bibitem{Li-Persson-Wang-Wu-14}
%B.~Li, T.~Persson, B.~Wang, and J.~Wu.
%\newblock Diophantine approximation of the orbit of 1 in the dynamical system
%  of beta expansions.
%\newblock {\em Math. Z.}, 276(3-4):799--827, 2014.

\bibitem{Lind_Marcus_1995}
D.~Lind and B.~Marcus.
\newblock {\em An introduction to symbolic dynamics and coding}.
\newblock Cambridge University Press, Cambridge, 1995.

%\bibitem{Lu-Wu-2016}
%F.~L\"{u} and J.~Wu.
%\newblock Diophantine analysis in beta-dynamical systems and {H}ausdorff
%  dimensions.
%\newblock {\em Adv. Math.}, 290:919--937, 2016.

\bibitem{Mauldin_Williams_1988}
R.~D. Mauldin and S.~C. Williams.
\newblock Hausdorff dimension in graph directed constructions.
\newblock {\em Trans. Amer. Math. Soc.}, 309(2):811--829, 1988.

\bibitem{Parry_1960}
W.~Parry.
\newblock On the $\beta$-expansions of real numbers.
\newblock {\em Acta Math. Acad. Sci. Hungar.}, 11:401--416, 1960.

\bibitem{Raith-89}
P.~Raith.
\newblock Hausdorff dimension for piecewise monotonic maps.
\newblock {\em Studia Math.}, 94(1):17--33, 1989.

\bibitem{Renyi_1957}
A.~R\'{e}nyi.
\newblock Representations for real numbers and their ergodic properties.
\newblock {\em Acta Math. Acad. Sci. Hungar.}, 8:477--493, 1957.

%\bibitem{Schmeling-97}
%J.~Schmeling.
%\newblock Symbolic dynamics for {$\beta$}-shifts and self-normal numbers.
%\newblock {\em Ergodic Theory Dynam. Systems}, 17(3):675--694, 1997.

\bibitem{Sidorov_2003}
N.~Sidorov.
\newblock Almost every number has a continuum of {$\beta$}-expansions.
\newblock {\em Amer. Math. Monthly}, 110(9):838--842, 2003.

\bibitem{Urbanski_1986}
M.~Urba{\'n}ski.
\newblock On {H}ausdorff dimension of invariant sets for expanding maps of a
  circle.
\newblock {\em Ergodic Theory Dynam. Systems}, 6(2):295--309, 1986.

\bibitem{Urbanski-87}
M.~Urba\'{n}ski.
\newblock Invariant subsets of expanding mappings of the circle.
\newblock {\em Ergodic Theory Dynam. Systems}, 7(4):627--645, 1987.

\end{thebibliography}

\end{document}